\documentclass[a4paper,10pt]{article}
\usepackage[utf8]{inputenc}
\usepackage{amsfonts,amssymb,amsmath,amsthm}
\usepackage{xfrac}
\usepackage{tikz-cd}
\usepackage{cite}
\usepackage{float}
\usepackage{fullpage}
\usepackage{hyperref}
\usepackage{authblk}

\hypersetup{colorlinks=true, linkcolor=red, urlcolor=blue, citecolor=blue, pdfpagemode=UseNone, pdfstartview=FitH, bookmarksopen=true}

\newcommand{\udp}{\raisebox{0.15ex}{\ensuremath{\diagup\kern-0.40em\searrow}}}
\newcommand{\dup}{\ensuremath{\diagdown\kern-0.40em\nearrow}}

\title{Jordan-Hölder with uniqueness for semimodular semilattices}
\author{Pavel Paták}
\author{Pavel Paták\footnote{The research stay of P.P. at IST Austria is funded by the project CZ.02.2.69/0.0/0.0/17\_050/0008466 Improvement of internationalization in the field of research and development at Charles University, through the
support of quality projects MSCA-IF.}}
\affil{Department of Applied Mathematics, Charles University, Prague,  Czech Republic}
\affil{IST Austria, Klosterneuburg, Austria.}

\newtheorem{theorem}{Theorem}
\newtheorem{lemma}[theorem]{Lemma}
\newtheorem{observation}[theorem]{Observation}

\theoremstyle{definition}
\newtheorem{definition}[theorem]{Definition}
\begin{document}

\maketitle

\begin{abstract}
We present a short proof of the Jordan-H\"older theorem with uniqueness for semimodular semilattice:
Given two maximal chains in a semimodular semilattice of finite height, they both have the same length. Moreover there is a unique bijection that takes the prime intervals of the first chain to the prime intervals of the second chain such that the interval and its image are up-and-down projective.

The theorem generalizes the classical result that all composition series of a finite group have the same length and isomorphic factors.
Moreover, it shows that the isomorphism is in some sense unique.
\end{abstract}

\section{Introduction}
The classical Jordan-Hölder theorem~\cite{Jordan1869,Jordan1870,Hoelder1889, Baumslag2006} tells us that any two composition series\footnote{A composition series is a maximal chain $1=G_0\vartriangleleft G_1\vartriangleleft G_2\vartriangleleft\ldots\vartriangleleft G_n=G$ of subgroups of $G$. The quotient groups $G_{i+1}/G_i$, are called its \emph{factors}. Any subgroup that appears in some composition series of $G$ is called \emph{subnormal}.} of a finite group have the same length and, up to a permutation, isomorphic factors. 
It is an essential structural result, which generalizes the fundamental theorem of arithmetics and allows to decompose each finite group into uniquely determined basic building blocks, called simple groups.
Thus in order to classify all finite groups it suffices to classify all simple groups; and the way they can be composed. 

One can easily extend Jordan-Hölder theorem to structures that extend groups, e.g. rings, modules, vector spaces or more generally, groups with operators. Given the importance of the theorem it is natural to ask whether the underlying group structure is needed for the theorem to hold. 

It was already clear to Dedekind that the first part of the statement is not specific to the lattice of subnormal groups. He observed that any two maximal chains in an arbitrary finite (semi)modular lattice\footnote{We note that the subnormal groups form a sublattice of the lattice of all subgroups~\cite{Wielandt1939}.
It is not hard to see that this lattice is dually semimodular, see for example~\cite[p.~302]{Stern1999}.} have the same length~\cite{Dedekind1900}.
However, a lattice theoretic generalization of the whole statement was only proven by Gr\"atzer and Nation in 2010~\cite{Graetzer2010}. To do so, they
introduced the concept of projectivity, a lattice theoretic analogue of the second isomorphism theorem in groups.
They showed that for any two maximal chains in a finite semimodular lattice there is a bijection that takes the prime intervals of one chain to the prime intervals of the second chain such that the interval and its image are up-and-down projective. In that form the range of applications extends from groups to much broader class of structures which for example include matroids\footnote{In matroids we do not have second isomorphism theorem, so up-and-down projectivity plays a smaller role. However, it still tells us something. For example, if $a$ and $b$ are points of a single matroid $M$, then $[0,a]\udp [0,b]$ in the lattice of flats if and only if $a$ and $b$ lie in the same component of $M$~\cite[Proofs of Theorems 393 and 396]{graetzer2011}.} and antimatroids.
The statement has been further generalized to semimodular posets~\cite{Ronse2018}.
In 2011 Cz\'edli and Schmidt~\cite{Czedli2011}
established the strongest form of the theorem for semimodular lattices by showing that for them the permutation is unique (Theorem~\ref{thm:jh}).

In their proof Cz\'edli and Schmidt compare the two chains by looking at the join semilattice generated by them and showing that it is planar. Then they use the theory of planar semimodular semilattices to deduce the result. Eventually, the ideas led to a  developed theory of planar semimodular lattices, which is very valuable by itself~\cite[Chapter~3]{Graetzer2014}.
However, for the proof of the uniqueness in Jordan-H\"older theorem the theory can be bypassed, which shortens the proof significantly.

Here we present a short, distilled proof of Jordan-H\"older theorem together with its uniqueness part,
based on the original inductive approach by Grätzer and Nation.

The paper is organized as follows.
First we recall some basic notions.
Then we introduce the concept of projectivity in a (semi)lattice, show several of its properties and compare it to the second isomorphism theorem in groups. 
After that we present the inductive proof of Theorem~\ref{thm:jh}.

\section{Preliminaries}
We recall the basic notions for reader not familar with (semimodular) lattices.

By a lattice we mean a poset $L$, where every two elements $a,b\in L$ have the least common bound, called \emph{meet} and denoted $a\wedge b$, and the greatest common bound, called \emph{join} and denoted $a\vee b$. The operations $\vee$ and $\wedge$ are commutative, associative, idempotent and satisfy the following absorption laws $(a\wedge b)\vee a =a$;  $(a\vee b)\wedge a =a$.
A poset $S$ is called \emph{join semilattice} if the greatest common bound $a\vee b$ exists for every two elements $a,b\in S$.

We write $a\preceq b$ , iff $a\leq b$ and there is no $c$ with $a<c<b$. If $a\preceq b$ and $a\neq b$, we write $a\prec b$.
An interval $[a,b]$ is called \emph{prime}, if $a\prec b$.
A join semilattice $L$ is called \emph{semimodular} if $a\preceq b$ implies $a\vee c\preceq b\vee c$ for every $c\in L$ and $a,b\in L$.
A \emph{chain} in a poset $P$ is a linearly ordered subset of $P$. 
A poset $P$ is of \emph{finite height}, if all its chains are finite. 

\section{Projectivity}
\begin{definition}
If $L$ is a lattice, we say that an interval $[a,b]$ is up-projective to $[x,y]$, written $[a,b]\nearrow [x,y]$, if and only if $b\wedge x=a$ and $b\vee x=y$. Equivalently we can write $[x,y]\searrow [a,b]$.

If $[a,b]$ and $[c,d]$ are two intervals and there is $[x,y]$ such that 
$[a,b]\nearrow [x,y]\searrow [c,d]$, we say that $[a,b]$ is \emph{up-and-down projective to $[c,d]$} and write $[a,b]\udp[c,d]$,
see Figure~\ref{fig:perspectives}. 
\end{definition}
\begin{figure}[H]
\begin{center}
\begin{minipage}{0.4\textwidth}
\begin{center}
 \begin{tikzcd}[tips=false]
  & y\ar[d]\ar[dl]\ar[dr] & \\
  b\ar[d] & x\ar[dl]\ar[dr]& d\ar[d] \\
  a & & c
 \end{tikzcd}\\[2ex]
 
 $[a,b]\nearrow[x,y]$, $[x,y]\searrow[c,d]$, \\ $[a,b]\udp[c,d]$
 \end{center}
 \end{minipage}\hfil
 \begin{minipage}{0.4\textwidth}
 \begin{center}
 \begin{tikzcd}[tips=false, row sep = 2ex]
  & & f\ar[dl]\ar[d]\\
  & d\ar[dl]\ar[d] & e\ar[dl]\\
  b\ar[d] & c\ar[dl]\\
  a
 \end{tikzcd}\\[2ex]
 
 $[a,b]\nearrow[c,d]\nearrow[e,f]$
 \end{center}
 \end{minipage}
\caption{Projective intervals}
\label{fig:perspectives}
\end{center}
\end{figure}
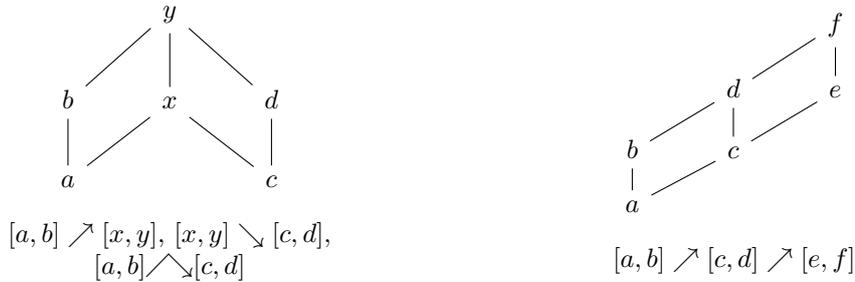

Let us now compare this notion to the second isomorphism theorem for groups.
If $L$ is the lattice of subgroups of some group, $[a,b]\nearrow[x,y]$ and $x$ is normal in $y$, the second isomorphism theorem tells us that $a$ is normal in $b$ and the quotient groups $b/a$ and $y/x$ are isomorphic.

\bigskip

We are going to use the following two  properties of projectivity.
\begin{lemma}\label{lem:primeperspectivity}
  Let $L$ be a lattice, $a,b, x,y \in L$ and $a \prec b$.
  Then $[a,b]\nearrow [x,y]$ if and only if
  $x\neq y$, $a\vee x = x$ and $b\vee x =y$.
 \end{lemma}
 \begin{proof}
 If $[a,b]\nearrow [x,y]$, then $b\vee x =y$ by definition and $a\vee x=(b\wedge x)\vee x = x$.
 The relation $a\prec b$ implies 
 $a=b\wedge x\neq b$, that is $b\nleq x$, and consequently $y=b\vee x\neq x$.
 
 If $x\neq y$, $a\vee x = x$, $b\vee x =y$, and $a\prec b$, then $a$ is a lower bound of $\{b,x\}$. Therefore, $a\leq b\wedge x\leq b$. By $a\prec b$ there is no $c\in L$ with $a<c<b$, so either $b\wedge x = b$, leading to the forbidden $x=b\vee x=y$, or $b\wedge x=a$, which together with $b\vee x=y$ shows $[a,b]\nearrow [x,y]$.
 \end{proof}

 Observe that for $a\prec b$, Lemma~\ref{lem:primeperspectivity} characterizes projectivity by joins only, hence, in the case $a\prec b$ and $c\prec d$, it allows us to extend the definition of $[a,b]\nearrow[c,d]$ and $[a,b]\udp[c,d]$ to join semilattices for the case $a\prec b$ and $c\prec d$.

\begin{observation}[Transitivity of $\nearrow$]\label{obs:perspTransitivity}
 Let $L$ be a semimodular join semilattice.
 If $a\prec b$, then $[a,b]\nearrow[c,d]$, $[c,d]\nearrow[e,f]$ implies
 $[a,b]\nearrow[e,f]$.
\end{observation}
\begin{proof}
First of all, the semimodularity and $c\neq d$ imply $c\prec d$. Thus we can use the semilattice definition for $[c,d]\nearrow [e,f]$ as  well.
Since $[c,d]\nearrow [e,f]$ we have $e\neq f$ and $e=c\vee e$, which implies $a\vee e=e$ and $b\vee e=f$, as required.
\end{proof}

\section{Jordan-H\"older theorem}
Finally we can state the main theorem.
\begin{theorem}[Jordan-Hölder]\label{thm:jh}
Let $L$ be an upper semimodular join semilattice of finite height.

Let $0=c_0\prec c_1\prec \ldots c_n=1$ and
$0=d_0\prec d_1\prec \ldots d_m=1$ be two maximal chains in $L$.
Then
\begin{enumerate}
 \item $m=n$.
 \item There is a unique permutation $\pi\in S_n$ such that\label{it:uniquePermutation}
 $[c_i,c_{i-1}]\udp [d_{\pi(i)},d_{\pi(i)-1}]$ for all $i=1,2,\ldots, n$.
 \item If $[c_i,c_{i-1}]\udp [d_{j},d_{j-1}]$, then $j\leq \pi(i)$, where $\pi$ is the same as in \ref{it:uniquePermutation}.\label{it:uniqueness}
\end{enumerate}
\end{theorem}
\begin{proof}
We proceed by induction on the height of $L$.
The statement is obviously true for height $0$ or $1$.
 
 So let the height of $L$ be higher and
 let $l$ be the largest integer such that
 $c_1\nleq d_l$. Clearly $l<m$.
 We set $d'_j:=c_1\vee d_j$ for all $j=0, \ldots, m$.
 Then $d'_0:=c_1$, $d'_l=d'_{l+1}=d_{l+1}$
 and $d'_j = d_j$ for $j\geq l+1$.
 Furthermore, we define $e_0=d_1$ and $e_i=d'_i$ for $i>0$, see Fig.~\ref{fig:induction}.
 
 \begin{figure}
  \begin{center}
   \includegraphics{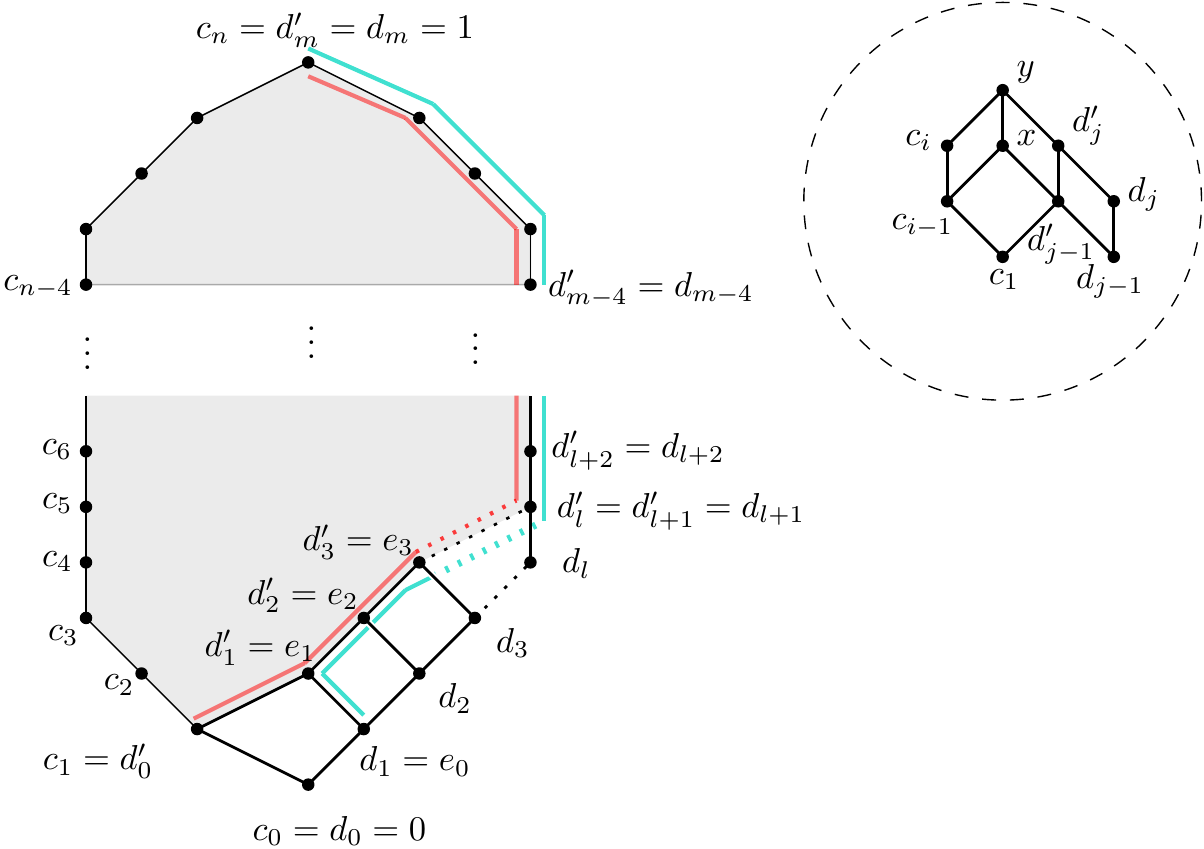}
    \caption{Illustration of the induction steps.}\label{fig:induction}
    \end{center}
 \end{figure}

 The ``red'' chain $c_1=d'_0\preceq d'_1\preceq\ldots\preceq d'_l = d'_{l+1}\preceq\ldots \preceq d'_m=1$ and the ``blue'' chain
 $d_1=e_0\preceq e_1\preceq\ldots\preceq e_l=e_{l+1}\preceq \ldots \preceq e_m=1$
 have obviously the same length and by semimodularity, they are maximal.
 By induction applied in $[c_1,1]$ the length of the red chain equals the length of $c_1\prec c_2\prec\ldots\prec c_n=1$.
 By induction in $[d_1,1]$, the length of the blue chain is the same as the length of $d_1\prec d_2\prec\ldots\prec d_m$. Thus $m=n$, and the first part of the theorem is proven.
 Consequently $d_0'\prec d_1'\prec \ldots d_l'=d_{l+1}'\prec d_{l+2}' \prec\ldots \prec d_m'$.
 
 Let us now find $\pi$. By induction in $[c_1,1]$, there is a unique bijection $\sigma\colon\{2,3,\ldots, n\}\to \{1,2,\ldots, l, l+2, l+3,\ldots, n\}$ such that
 $[c_{i-1},c_i]\udp[d'_{\sigma(i)-1},d'_{\sigma(i)}]$. 
 By construction, $[d'_{\sigma(i)-1},d'_{\sigma(i)}]\searrow[d_{\sigma(i)-1},d_{\sigma(i)}]$ and $[c_0,c_1]\nearrow [d_l,d_{l+1}]$. 
 Therefore, Observation~\ref{obs:perspTransitivity} implies $[c_{i-1},c_i]\udp [d_{\pi(i)-1},d_{\pi
 (i)}]$ for $i=1,\ldots, n$ if we set
 \[
  \pi(i):=\begin{cases}
           \sigma(i) & \text{if $i>1$,}\\
           l+1 & \text{if $i=1$.}
          \end{cases}
 \]

 We now prove that $[c_{i-1},c_i]\udp [d_{j-1},d_j]$ implies $j\leq \pi(i)$. This clearly implies the uniqueness of $\pi$. 
 
 So let $[c_{i-1},c_i]\nearrow [x,y]\searrow  [d_{j-1},d_j]$ for some $x,y\in L$. By Lemma~\ref{lem:primeperspectivity}, $x\neq y$. There are two cases:
 \begin{enumerate}
  \item[$i=1$] Then, $x\neq y=x\vee c_1$ implies $c_1\nleq x$. Thus $d_{j-1}\leq x$ gives $c_1\nleq d_{j-1}$, so $j-1\leq l$ by the definition of $l$.
  Hence $j\leq l+1 = \pi(1)$.
  
  \item[$i>1$] Then $[c_{i-1},c_i]\nearrow [x,y]$ implies $y>x\geq c_{i-1}\geq c_1$,
  see top right part of Figure~\ref{fig:induction}.
  From Lemma~\ref{lem:primeperspectivity}, we immediately obtain $x,y\in [c_1,1]$, $x\neq y$, $x\vee d'_{j-1} = x\vee (c_1\vee d_{j-1}) = (x\vee c_1) \vee d_{j-1} = x\vee d_{j-1} = x$; and $x\vee d'_j = x\vee (c_1\vee d_j) = (x\vee c_1)\vee d_j = x\vee d_j = y$. By Lemma~\ref{lem:primeperspectivity} this implies that in $[c_1,1]$ one has $[c_{i-1},c_i]\udp [d'_{j-1},d'_j]$.
  So by induction hypothesis $j\leq\sigma(i)=\pi(i)$, which finishes the proof.\qedhere
\end{enumerate}
\end{proof}

\bibliographystyle{alpha}
\bibliography{jh}
\end{document}